\documentclass[12pt]{article}
\usepackage[utf8]{inputenc}
\usepackage[T1]{fontenc}
\usepackage{tikz}
\usepackage{lmodern, dsfont, pifont}
\usepackage[round, sort]{natbib}
\usepackage[french, english]{babel}
\usepackage{xcolor,graphicx,psfrag,tcolorbox}
\definecolor{Blue}{rgb}{0.1,0.1,0.5}
\usepackage[hypertexnames=false]{hyperref}
\hypersetup{bookmarksnumbered,
            colorlinks=true,
            citecolor=Blue,
            linkcolor=black,
            breaklinks,
            pdfborder=0 0 0,
            pdftitle=,
            pdfauthor=,
            pdfkeywords=}

\usepackage{amsmath,amsthm,amsfonts,amssymb,bm}

\newtheorem{theorem}{Theorem}[section]
\newtheorem{corollary}[theorem]{Corollary}

\theoremstyle{remark}
\newtheorem{remark}{Remark}

\usepackage{graphicx}
\DeclareMathOperator{\Var}{var}

\DeclareMathOperator{\argmax}{argmax}

\usepackage{xcolor}	 
\definecolor{myred}{RGB}{255,66,56}

\newcommand   \one    {\mathds{1}}
\newcommand   \Esp    {\mathbb{E}}
\newcommand   \Nset   {\mathbb{N}}
\newcommand   \Rset   {\mathbb{R}}
\newcommand   \dx     {\mathrm{d}x}
\newcommand   \dy     {\mathrm{d}y}

\newcommand \SnMC     {\hat{S}_{n}^{\rm MC}}
\newcommand \SnCV     {\hat{S}_{n}^{\rm cv}}
\newcommand \SnSt     {\hat{S}_{n}^{\rm str}}

\newcommand \DDelta   {\bm{\Delta}}
\newcommand \uDelta   {\underline{\Delta}}
\newcommand \FD       {F_{\uDelta}}

\begin{document}

\begin{center}
  \Large\linespread{1.0}\textsc{%
    Integration of bounded monotone functions: %
    Revisiting the nonsequential case, with a focus on unbiased Monte
    Carlo (randomized) methods}

\normalsize
\bigskip
Subhasish Basak$^{1, 2, \star}$ \& Julien Bect$^{2}$ \& Emmanuel Vazquez$^{2}$ 
\bigskip

{\it \selectlanguage{french} %
  $^{1}$ ANSES, Maison-Alfort, France.
  $^\star$E-mail: \texttt{subhasish.basak@centralesupelec.fr}\\
 $^{2}$ Université Paris-Saclay, CNRS, CentraleSupélec,\\
 Laboratoire des signaux et systèmes, Gif-sur-Yvette, France.}
\end{center}
\bigskip


\selectlanguage{french}

{\bf R\'esum\'e.} %
Dans cet article, nous revisitons le problème de l'intégration
numérique d'une fonction monotone bornée, en nous concentrant sur la
classe des méthodes de Monte Carlo non séquentielles. %
Nous établissons dans un premier une borne inférieure pour l'erreur
maximale dans~$L^p$ d'un algorithme non séquentiel, qui généralise
pour~$p > 1$ un théorème de Novak. %
Nous étudions ensuite, dans le cas~$p = 2$, l'erreur maximale de deux
méthodes sans biais---une méthode fondée sur l'utilisation d'une
variable de contrôle, et la méthode de l'échantillonnage stratifié.

{\bf Mots-cl\'es.} %
Intégration monotone, %
Monte Carlo, %
échantillonnage stratifié, %
échantillonnage hypercube latin, %
variable de contrôle, %
complexité

\medskip \selectlanguage{english}

{\bf Abstract.} %
In this article we revisit the problem of numerical integration
for monotone bounded functions, %
with a focus on the class of nonsequential Monte Carlo methods. %
We first provide new a lower bound on the maximal $L^p$ error of
nonsequential algorithms, %
improving upon a theorem of Novak when~$p > 1$. %
Then we concentrate on the case $p = 2$ and study the maximal error of
two unbiased methods---namely, a method based on the control variate
technique, and the stratified sampling method.


\bigskip\bigskip

\section{Introduction}
\label{sec:intro}

We address in this article the problem of constructing a numerical
approximation of the
expectation~$\Esp(g(Y)) = \int g(y)\, \mathrm{P}_Y(\dy)$, where $Y$~is
a real random variable with known distribution~$\mathrm{P}_Y$, and
$g$~is a real function that is bounded and monotone. %
Such a problem occurs naturally in applications where one is
interested in computing a risk using a model that provides an
increasing conditional risk~$g(Y)$ with respect to some random
variable~$Y$. %
This situation occurs for instance in the field of food safety, with
$Y$ a dose of pathogen and $g(Y)$~the corresponding probability of
food-borne illness \cite[see, e.g.,][]{perrin14}.

Assuming that the cumulative distribution function of~$Y$ is
continuous, the problem reduces after change of variable and scaling
to the computation of $S(f) = \Esp(f(X)) = \int_0^1 f(x)\, \dx$, %
where $X$ is uniformly distributed on~$\left[ 0, 1 \right]$ and
$f$~belongs to the class~$F$ of all non-decreasing functions defined
on~$\left[ 0, 1 \right]$ and taking values in~$\left[ 0, 1 \right]$. %
We work in this article in a fixed sample-size setting, where the
number $n$ of evaluations of~$f$ to be performed is chosen beforehand.

This problem was first studied by \cite{kiefer57}, who proved that
considering regularly-spaced evaluations at~$x_i = i/(n+1)$,
$1 \le i \le n$, and then using the trapezoidal integration rule
assuming~$f(0)=0$ and~$f(1) = 1$, is optimal in the worst-case sense
among all deterministic---possibly sequential---methods, %
with maximal error~$1/(2(n+1))$. %
\cite{novak92} later studied Monte Carlo (a.k.a.\ randomized) methods,
and established that sequential methods are better in this setting
than nonsequential ones, with a minimax rate of~$n^{-3/2}$ over~$F$
for the~$L^1$ error. %
Novak's proof relies on the construction of a particular two-stage
algorithm, using stratified sampling in the second stage.

This article revisits the nonsequential setting with a focus on
unbiased Monte Carlo methods, which are a key building block for the
construction of good (rate-optimal) sequential methods---as can be
learned from the proof of Theorem~3 in Novak's article. %
Section~\ref{sec:bound} derives a lower bound for the maximal
$L^p$-error of nonsequential methods, for any~$p\ge 1$, which is a
generalization of a result by \cite{novak92} concerning the~$L^1$
error. %
Sections~\ref{sec:cv} and \ref{sec:strata} then study the maximal
$L^2$~error (variance) of two simple unbiased methods, based
respectively on the control variate technique and on stratification. %
Section~\ref{sec:con} concludes the article with a discussion.

\section{A lower bound for the maximal $L^p$ error}
\label{sec:bound}

A nonsequential (also called non-adaptive) Monte Carlo method first
evaluates the function at $n$~random points $X_1$, \ldots, $X_n$
in~$[0, 1]$, %
and then approximates the integral~$S(f)$ using an estimator
\begin{equation}
  \hat{S}_n(f) = \varphi \left( X_1,\, f(X_1),\, \hdots,\, X_n,\, f(X_n) \right),
\end{equation}
where $\varphi: \left[ 0, 1 \right]^{2n} \rightarrow \mathbb{R}$ is a measurable
function. %
A nonsequential method is thus defined by two ingredients: the
distribution of~$(X_1, \ldots, X_n)$ and the function~$\varphi$. %
The worst-case $L^p$~error of such a method over the class~$F$ is
\begin{equation}
  e_p(\hat{S}_n) \;=\; \sup_{f \in F}\,
  \Esp\left( \left| S(f) - \hat{S}_n(f) \right|^p \right)^{1/p}.
\end{equation}

\begin{remark}
  The class of nonsequential Monte Carlo methods as usually defined in
  the literature also allows~$\hat S_n$ to be randomized (i.e.,
  allows~$\varphi$ to be a random function). %
  We have not considered randomized estimators in our definition,
  however, since Rao-Blackwell's theorem implies that they do not help
  in this setting, for any convex loss function.
\end{remark}

\medskip

\cite{novak92} proved that for any nonsequential Monte Carlo method
with sample size~$n$, the maximal $L^1$~error $e_1(\hat{S}_n)$ is
greater or equal to~$1/(8n)$. %
We generalize this result to the case of the $L^p$~error.

\medbreak

\begin{theorem}\label{th:lp}
  For any nonsequential Monte Carlo methods with sample size~$n$,
  \begin{equation*}
    e_p(\hat{S}_n) \;\geq\;
    \left(\frac{1}{2} \right)^{2+1/p}\, \frac{1}{n}\, .
  \end{equation*}
\end{theorem}

\medbreak

Observe that Novak's lower bound is recovered for~$p = 1$. %
Using Theorem~\ref{th:lp} with $p=2$ we can deduce of lower bound for
the variance of unbiased nonsequential methods.

\medbreak

\begin{corollary}
  \label{cor:ub_low_bdd}
  For any unbiased nonsequential Monte Carlo method with sample size~$n$,
  \begin{equation*}
    \sup_{f \in F}\, \Var\left( \hat{S}_{n}(f) \right) \;\geq\; \frac{1}{32n^2}.
  \end{equation*}
\end{corollary}

\medbreak

\begin{proof}[Proof of Theorem~\ref{th:lp}]
  Consider a nonsequential Monte Carlo methods with evaluation
  points~$X_1$, \ldots, $X_n$ and estimator~$\hat S_n$.
  Divide the interval $[0, 1]$ into $2n$ equal subintervals of length
  $1/(2n)$: %
  then at least one of the subintervals, call it~$I$, will contain no
  evaluation point with probability at least~$1/2$.
  Now construct two functions $f_1, f_2 \in F$ that are both equal to
  zero on the left of~$I$, equal to one on the right, and such that
  $f_1 = 1$ and $f_2 = 0$ on $I$. %
  Then $S(f_1) - S(f_2) = 1/(2n)$, and $\hat S_n(f_1) = \hat S_n(f_2)$
  on the event $A = \{\{X_1, X_2, \hdots, X_n\} \cap I = \emptyset\}$,
  since $f_1$~and~$f_2$ coincide outside of~$I$. %
  It follows that
  \begin{align*}
    \left( e_p(\hat{S}_n) \right)^p
    & \;\geq\; \sup_{f\in \{ f_1, f_2 \}}\, \Esp\bigl( | S(f)- \hat{S}_n(f) |^p \bigr)
      \;\geq\; \frac{1}{2}\, \sum_{j=1}^2 \Esp\bigl( |S(f_j)-\hat{S}_n(f_j) |^p \bigr) \\
    & \;\geq\; \frac{1}{2}\, \sum_{j=1}^2 \Esp(|S(f_j)-\hat{S}_n(f_j)|^p \cdot \one_{A})
      \;=\; \frac{1}{2} \sum_{j=1}^2 \Esp(|S(f_j)-T|^p \cdot \one_{A}),
  \end{align*}
  where $T$ denotes the common value of~$\hat S_n(f_1)$
  and~$\hat S_n(f_2)$ on~$A$. %
  We conclude that
  \begin{equation*}
    \left( e_p(\hat{S}_n) \right)^p
    \;\geq\; \frac{1}{2}\; \frac{|I|^p}{2^{p-1}}\; P(A)
    \;\geq\; \big(\frac{1}{2}\big)^{2p+1} \cdot \frac{1}{n^p},
  \end{equation*}
  using the fact that, for any $a, b, x \in \mathbb{R}$ and~$p\ge 1$,
  $|a-x|^p + |b-x|^p \geq |a-b|^p / 2^{p-1}$.
\end{proof}

\section{Uniform i.i.d.\ sampling}
\label{sec:cv}

The simple Monte Carlo method is the most common example of a
nonsequential method: %
the evaluation points $X_1$, \ldots, $X_n$ are drawn independently,
uniformly in~$\left[ 0, 1 \right]$, and then the integral is estimated
by $\SnMC(f) = \frac{1}{n}\sum_{i=1}^n f(X_i)$. %
The estimator is clearly unbiased, and it follows from Popoviciu's
inequality---i.e., $\Var(Z) \le 1/4$ for any random variable~$Z$
taking values in~$\left[0, 1 \right]$---that
\begin{equation*}
  \left( e_2\left( \SnMC \right) \right)^2
  \;=\; \max_{f \in F}\, \Var\left( \SnMC(f) \right)
  \;=\; \frac{1}{4n}.
\end{equation*}
The maximal error is attained when $f$~is a unit step function jumping
at~$x_0 = 1/2$. %
It turns out that a smaller error can be achieved, for the same
(uniform i.i.d.) sampling scheme, using the control variate
technique. %
More specifically, we consider the control variate
$\tilde{f}(X_i) = X_i$ and set
\begin{equation*}
  \SnCV(f) \;=\; \frac{1}{n}\,
  \sum_{i=1}^n \left( f(X_i) - \tilde{f}(X_i) \right)
  \,+\, \frac{1}{2}.  
\end{equation*}

\medbreak

\begin{theorem}
  \label{th:cv}
  The estimator~$\SnCV(f)$ is unbiased, and satisfies
  \begin{equation*}
    \left( e_2\left( \SnCV \right) \right)^2
    \;=\; \max_{f \in F}\, \Var\left( \SnCV(f) \right)
    \;=\; \frac{1}{12n}.
  \end{equation*}
  The maximal error is attained for any unit step function.
\end{theorem}

\medbreak 

\begin{proof}
  The estimator is unbiased since $\Esp(\tilde{f}(X_i)) = 1/2$, and
  therefore the mean-squared error is equal to
  $\Var(\SnCV(f)) = \frac{1}{n}\Var(f(X)-X)$. %
  For a unit step function $f = \one_{\left[ x_0, 1 \right]}$ with a
  jump at $x_0 \in [0, 1]$, the random variable $f(X) - X$ is
  uniformly distributed over $\left[-x_0, 1-x_0\right]$, which yields
  $\Var(\SnCV(f)) = 1/(12n)$ as claimed. %
  It remains to show that $\Var(f(X)-X) \le \frac{1}{12}$ for
  all~$f \in F$.

  Let $F_m \subset F$ denote the class of all non-decreasing staircase
  functions of the form
  $f = \sum_{k=1}^m \alpha_k \cdot \one_{(\frac{k-1}{m},
    \frac{k}{m}]}$, with
  $0 \leq \alpha_1 \leq \alpha_2 \leq \hdots \leq \alpha_m \leq 1$.
  For any $f \in F$, consider the piecewise-constant approximation
  $f_m \in F_m$ defined by averaging~$f$ over each subinterval of
  length~$1/m$. Then, $\Esp(f_m(X)) = \Esp(f(X))$ and
  $\left| \Var\left( f_m(X) - X \right) - \Var\left( f(X) - X \right)
  \right| \le \frac{1}{m}$. Thus,
  \begin{equation}
    \label{eq:apprx}
    \sup_{f \in F} \Var\left( f(X) - X \right)
    \;=\; \lim_{m \to \infty} 
    \sup_{f \in F_m} \Var\left( f(X) - X \right).
  \end{equation}

  Let us now show that $\Var\left( f(X) - X \right)$ is maximized
  over~$F_m$ when $f$~is a unit step function. %
  Pick any
  $f = \sum_{k=1}^m \alpha_k \cdot \one_{(\frac{k-1}{m}, \frac{k}{m}]}
  \in F_m$. %
  Set $\alpha_0 = 0$ and $\alpha_{m+1} = 1$. %
  If $f$ is not a unit step function, then there exist
  $k_1, k_2 \in \{ 1, \ldots, m \}$ such that $k_1 \le k_2$ and
  $\alpha_{k_{1}-1} < \alpha_{k_1} = \hdots = \alpha_{k_2} <
  \alpha_{k_{2}+1}$. %
  Denote by~$f_u \in F_m$ the function obtained by changing
  the
  common value of~$\alpha_{k_1}$, \ldots, $\alpha_{k_2}$ in~$f$ to~$u \in \left[ \alpha_{k_{1}-1}, \alpha_{k_{2}+1} \right]$. %
  The variance $\Var(f_u(X)-X)$ is a convex function of~$u$,
  since it can be
  expanded as $au^2 + bu + c$ with
  $a = \frac{k_2-k_1+1}{m} (1-\frac{k_2-k_1+1}{m}) >0$. %
  Consequently, we have $\Var(f_u(X)-X) > \Var(f(X)-X)$ at one of the
  two endpoints
  of~$\left[ \alpha_{k_{1}-1}, \alpha_{k_{2}+1} \right]$. %
  Note that the corresponding staircase function~$f_u$ has one fewer
  step than~$f$. %
  Iterating as necessary, we conclude that for any~$f \in F_m$ there
  exists a unit step function~$g \in F_m$ such that
  $\Var\left( f(X) - X\right) \le \Var\left( g(X) - X \right) =
  \frac{1}{12}$. %
  Therefore
  $\sup_{f \in F_m} \Var\left( f(X) - X \right) = \frac{1}{12}$, which
  completes the proof.
\end{proof}

\section{Stratified sampling}
\label{sec:strata}

Consider now a stratified sampling estimator with $K$~strata:
\begin{equation}\label{equ:SnSt-def}
  \SnSt(f) \;=\;
  \sum_{k=1}^K w_k \cdot \frac{1}{n_k}
  \sum_{i=1}^{n_k} f\left( X_{k,i} \right),
\end{equation}
where the $k$-th stratum is $I_k = \left[ x_{k-1}, x_k \right]$,
$0 = x_0 < x_1 < \cdots < x_{K-1} < x_K = 1$, %
the weight $w_k = \left| x_{k-1} - x_k \right|$ is the length of the
$k$-th stratum, %
the allocation scheme $(n_1, \ldots, n_K)$ is such that $n_k \ge 0$
for all~$k$ and $\sum_k n_k = n$, %
and the random variables~$X_{k,i}$ are independent, %
with the $X_{k,i}$s uniformly distributed in~$I_k$. %
Note that the sampling points are no longer identically distributed
here. %
The estimator $\SnSt(f)$ is unbiased, with variance
\begin{equation}\label{equ:SnSt-var}
  \Var\left( \SnSt(f) \right) \;=\;
    \sum_{k=1}^K \frac{w_k^2}{n_k}\, \Var\left( f\left( X_{k,1} \right) \right).
\end{equation}

\begin{theorem}
  \label{th:st}
  For any $K \le n$, any choice of strata and any allocation scheme,
  the stratified sampling estimator~\eqref{equ:SnSt-def} satisfies
  \begin{equation}\label{equ:SnSt-max}
    \left( e_2\left( \SnSt \right) \right)^2
    \;=\; \max_{f \in F}\, \Var\left( \SnSt(f) \right)
    \;=\; \frac{1}{4}\, \max_k \frac{w_k^2}{n_k},    
  \end{equation}
  The maximal error is attained for a unit step function with a jump
  at the middle of~$I_{k^*}$, where $k^* \in \argmax w_k^2 / n_k$. %
  The minimal value of the maximal error~\eqref{equ:SnSt-max} is
  $\frac{1}{4n^2}$, and is obtained with $K = n$ strata of equal
  lengths ($w_k = n^{-1}$ and $n_k = 1$ for all~$k$).
\end{theorem}

The optimal stratified sampling method can be seen as a
one-dimensional special case of the Latin Hypercube Sampling (LHS)
method \citep{mckay79}. %
(On a related note, \cite{mckay79} prove that, in any dimension, the
LHS method is preferable to the simple Monte Carlo method if the
function is monotone in each of its arguments.)

\begin{proof}
  For all~$K \in \Nset^*$, let
  $\DDelta_K = \left\{ \left( \Delta_1,\, \ldots,\, \Delta_K
    \right) \in\Rset_+^K \mid \sum_{k=1}^K \Delta_k \le 1
  \right\}$. %
  For a given stratified sampling method with $K$~strata, for all
  $\uDelta \in \DDelta_K$, define
  \begin{equation*}
    \FD \;=\; \left\{
      f \in F \,\bigm|\, %
      \forall k \in \{ 1, \ldots, K \},\, %
      f(x_k) - f(x_{k-1}) = \Delta_k
    \right\}.
  \end{equation*}
  Then it follows from~\eqref{equ:SnSt-var} and Popoviciu's inequality
  that
  \begin{equation}\label{equ:max-f-ND}
    \max_{f \in \FD}\, \Var\left( \SnSt(f) \right) \;=\;
    \frac{1}{4}\, \sum_{k=1}^K \frac{w_k^2 \Delta_k^2}{n_k},
  \end{equation}
  where the maximum is attained for a non-decreasing staircase
  function with jumps of height~$\Delta_k$ at the middle of the
  strata. %
  Note that $\sum_{k=1}^K \Delta_k^2 \leq \sum_{k=1}^K \Delta_k \leq 1$.
  Therefore, the right-hand side of~\eqref{equ:max-f-ND} is
  upper-bounded by $\frac{1}{4}\, \max_k w_k^2 / n_k$, %
  which is indeed the value of the variance~\eqref{equ:SnSt-var} when
  $f$ is a unit step function with a jump at the middle of the stratum
  where $w_k^2 / n_k$ is the largest.

  In order to prove the second part of the claim, observe that
  any stratum with~$n_k \ge 2$ can be further divided into~$n_k$
  sub-strata of equal lengths without increasing the upper bound. %
  Considering then the case where $K = n$ and $n_k = 1$ for all~$k$,
  the upper bound reduces to~$\frac{1}{4}\, \max_k w_k^2$, which is
  minimal when $w_1 = \cdots = w_n = n^{-1}$ since $\sum_k w_k = 1$.
\end{proof}

\section{Discussion}
\label{sec:con}

The stratified sampling (LHS) method provides the best-known variance
upper bound over the class~$F$ for an unbiased nonsequential method as
soon as~$n \ge 3$, but is outperformed by the control variate method
of Section~\ref{sec:cv} when~$n \le 2$. %
We do not know at the moment if these results are optimal in the class
of unbiased nonsequential methods. %
(The ratio between the best variance upper bound and the lower bound
of Corollary~\ref{cor:ub_low_bdd} is $\frac{8}{3} \approx 2.67$
for~$n=1$, $\frac{16}{3} \approx 5.33$ for~$n = 2$ and $8$
for~$n \ge 3$.)

Relaxing the unbiasedness requirement, it turns out that both methods
are outperformed for all~$n$ by the (deterministic) trapezoidal
method discussed in the introduction, which has a worst-case squared
error of $1/(4(n+1)^2)$. %
The ratio of worst-case mean-squared errors, however, is never very
large---at most~$\frac{16}{9} \approx 1.78$---and goes to~$1$ when~$n$ goes to infinity.

Directions for future work include constructing better sequential
methods using the findings of this article, and extending our results
to mutivariate integration with partial monotonicity \citep[see,
e.g.,][Section~2.1]{mckay79}.

\bibliographystyle{unsrtnat}
\bibliography{sbasak-monotone-nonseq.bib}

\newpage

\begin{tcolorbox}
  This work is part of the ArtiSaneFood project (ANR-18-PRIM-0015),
  which is part of the PRIMA program supported by the European Union.
\end{tcolorbox}

\end{document}